\documentclass[a4paper]{amsart}

\usepackage{amsmath}
\usepackage{amsfonts}
\usepackage{amssymb}

\newtheorem{thm}{Theorem}[section]
\newtheorem{lmm}[thm]{Lemma}
\newtheorem{prp}[thm]{Proposition}
\newtheorem{cor}[thm]{Corollary}

\theoremstyle{definition}
\newtheorem{dfn}[thm]{Definition}
\newtheorem{exm}[thm]{Example}

\theoremstyle{remark}
\newtheorem{rmr}[thm]{Remark}

\newcommand{\f}{\varphi}
\newcommand{\skp}[2]{\left<#1,#2\right>}
\newcommand{\id}{\mathop{\mathrm{id}}\nolimits}
\newcommand{\ind}{\mathop{\mathrm{ind}}\nolimits}
\newcommand\G{\mathrm{I}\!\mathrm{\Gamma}}
\renewcommand\L{\mathcal L}
\newcommand\K{\mathcal K}
\newcommand\e{\varepsilon}

\numberwithin{equation}{section}

\begin{document}

\title{On the index of product systems of Hilbert modules}

\author{Dragoljub J.\ Ke\v cki\'c}
\address{Faculty of Mathematics, University of Belgrade, Student\/ski trg 16-18, 11000 Beo\-grad, Serbia}
\email{keckic@matf.bg.ac.rs}

\author{Biljana Vujo\v sevi\'c}
\address{Faculty of Mathematics, University of Belgrade, Student\/ski trg 16-18, 11000 Beo\-grad, Serbia}
\email{bvujosevic@matf.bg.ac.rs}


\date{}



\begin{abstract}In this note we prove that the set of all uniformly continuous units on a product system over
a $C^*$ algebra $\mathcal B$ can be endowed with the structure of left right $\mathcal B$ - $\mathcal B$
Hilbert module after identifying similar units by the suitable equivalence relation. We use this construction
to define the index of the initial product system, and prove that it is the generalization of earlier defined
indices by Arveson (in the case $\mathcal B=\mathrm C$) and Skeide (in the case of spatial product system).
We prove that such defined index is a covariant functor from the category od continuous product systems to
the category of $\mathcal B$ bimodules. We also prove that the index is subadditive with respect to the outer
tensor product of product systems, and prove additional properties of the index of product systems that can
be embedded into a spatial one.
\end{abstract}

\maketitle

\section{Introduction}

Product systems over $\mathbb C$ have been studied during last several decades in connection with
$E_0$-semigroups acting on a type $I$ factor. Although the main problem of classification of all non
isomorphic product systems is still open, this theory is well developed. The reader is referred to book
\cite{Arv03} and references therein. In the present century there are some significant results that
generalizes this theory to product systems over some $C^*$-algebra $\mathcal B$, either in connection with
$E_0$ semigroups acting on $II_1$ factor (see \cite{JOT04}) or in connection with quantum probability
dynamics (see \cite{BS00}, \cite{JFA04}, \cite{SK03}).

There are many difficulties in generalizing the notion of index of a product system introduced in
\cite{Arv89} to this more general concept. Up to our knowledge there are two attempts in this direction. The
first was done in \cite{KPR06}, using redefining the notion of tensor product of two product systems in order
to retain well behaviour of index with respect to tensor product. The other attempt \cite{BJM11}, \cite{Bhat}
also uses the new concept of products, amalgamated product.

The main point of this note is to find the natural generalization of the index of product systems from
Arveson's $\mathbb C$-case to more general $C^*$-algebra case. In this purpose we consider the quotient set
$\mathcal U/\sim$ of all uniformly continuous units on the given product system $E$ by a suitable equivalence
relation, and prove that $\mathcal U/\sim$ carries a natural structure of two sided $\mathcal B - \mathcal B$
module.

Throughout the whole paper $\mathcal B$ will denote a unital $C^*$-algebra and $1$ will denote its unit.
Also, we shall use $\otimes$ for tensor product, either algebraic or other, although $\odot$ is also in
common use.

The rest of the introduction is devoted to basic definitions.

\begin{dfn} a) Product system over $C^*$-algebra $\mathcal B$ is a family $(E_t)_{t\ge0}$ of Hilbert $\mathcal
B-\mathcal B$ modules, with $E_0\cong\mathcal B$, and a family of (unitary) isomorphisms
$$\f_{t,s}:E_t\otimes E_s\to E_{t+s},$$
where $\otimes$ stands for the so called inner tensor product obtained by identifications $u b\otimes v\sim
u\otimes bv$, $u\otimes vb\sim(u\otimes v)b$, $bu\otimes v\sim b(u\otimes v)$, ($u\in E_t$, $v\in E_s$,
$b\in\mathcal B$) and then completing in the inner product $\skp{u\otimes v}{u_1\otimes v_1}=\skp v{\skp
u{u_1}v_1}$;

b) Unit on $E$ is a family $u_t\in E_t$, $t\ge0$, such that $u_0=1$ and $\f_{t,s}(u_t\otimes u_s)=u_{t+s}$,
which we shall abbreviate to $u_t\otimes u_s=u_{t+s}$. A unit $u_t$ is unital if $\skp{u_t}{u_t}=1$. It is
central if for all $b\in\mathcal B$ and all $t\ge0$ there holds $bu_t=u_tb$;
\end{dfn}

The previous definition does not include any technical condition, such as measurability, or continuity. It
occurs that it is sometimes more convenient to pose the continuity condition directly on units, although
there is a definition of continuous product system which we shall use in Section \ref{secdefind}.

\begin{dfn} Two units $u_t$ and $v_t$ give rise to the family of mappings $\K^{u,v}_t:\mathcal B\to\mathcal B$,
given by
$$\K^{u,v}_t(b)= \skp{u_t}{bv_t}.$$
All $\K^{u,v}_t$ are bounded $\mathbb C$-linear operators on $\mathcal B$, and this family forms a semigroup.
We say that the set of units $S$ is continuous if the corresponding semigroup $(K_t^{\xi,\eta})_{\xi,\eta\in
S}$ (with respect to Schur multiplying) is uniformly continuous. For a single unit $u_t$ we say that it is
uniformly continuous, or briefly just continuous, if the set $\{u\}$ is continuous, that is, the
corresponding family $\K^{u,u}_t$ is continuous in the norm of the space $L(\mathcal B)$.
\end{dfn}

Given a (uniformly) continuous set of units $\mathcal U$, we can form, as it was shown in \cite{JFA04} a
uniformly continuous completely positive definite semigroup ($CPD$-semigroup in further)
$\mathcal{K}=(\mathcal{K}_t)_{t\in{\mathbb{R}}_{+}}$

Denote by $\mathcal{L}=\frac{d}{dt}\mathcal{K} \mid _{t=0}$ the generator of CPD-semigroup $\mathcal{K}$. It
is well known \cite{JFA04} that $\mathcal{L}$ is conditionally completely positive definite, that is, for all
finite $n$-tuple $x_1,\dots,x_n\in\mathcal U$ and for all $a_j$, $b_j\in\mathcal B$ we have
\begin{equation}\label{CCPD}
\sum_{j=1}^na_jb_j=0\Longrightarrow\sum_{i,j=1}^nb_i^*\mathcal L^{x_i,x_j}(a_i^*a_j)b_j\ge0.
\end{equation}
It holds, also, $\mathcal L^{y,x}(b)=\mathcal L^{x,y}(b^*)^*$.

Also, $\K$ is uniquely determined by $\mathcal L$. More precisely, $\K$ can be recovered from $\mathcal L$ by
$\K=e^{t\mathcal L}$ using Schur product, i.e.
\begin{equation}\label{K preko L}
\K^{x,y}_t(b)=\skp{x_t}{by_t}=(\exp t\mathcal L^{x,y})(b).
\end{equation}

\begin{rmr}It should distinct the continuous set of units and the set of continuous units. In the second case
only $\K^{\xi,\xi}_t$ should be uniformly continuous for $\xi\in S$, whereas in the first case all
$\K_t^{\xi,\eta}$ should be uniformly continuous.
\end{rmr}

In Section \ref{secprel} we list and prove auxiliary statements that are necessary for the proofs of main
results. In Section \ref{secdefind} we define the notion of index of a given product system and prove its
functoriality from the category of continuous product system to the category of left-right $\mathcal B$ -
$\mathcal B$ Hilbert modules. Section \ref{sectensor} is devoted to outer tensor product of product systems
and to the behaviour of the index with respect to it. In Section \ref{secspatial} we discuss how the
existence of a central unit, either in the product system $E$ or in some its extension, affects the index.
All examples are left for the last Section \ref{secremarks}, as well as concluding remarks.

The proofs in this note requires technique specific for Hilbert $C^*$-modules reduced to a few initial
statements. Nevertheless, we refer the reader to books \cite{MT} and \cite{Lance} for elaborate approach to
this topic.

\section{Preliminary results}\label{secprel}

In \cite{PAMS08}, Liebscher and Skeide introduce an interesting way to obtain new units in a given product
system. The results are stated in Lemma 3.1, Proposition 3.3 and Lemma 3.4 of the mentioned paper. We briefly
quote them as

\begin{prp}\label{LS}a) Suppose that a continuous set $S$ of units generates a product system $E$. If $t\mapsto y_t\in
E_t$ is a mapping (not necessarily unit), with infinitesimal generators $K$ and $K_\xi\in L(\mathcal B)$
($\xi\in S$) in the sense that for all $b\in\mathcal B$ we have
$$\skp{y_t}{b y_t}=b+tK(b)+O(t^2),$$
$$\skp{y_t}{b\xi_t}=b+tK_\xi(b)+O(t^2).$$
Then there exists a product system $F\supseteq E$ and a unit $\zeta$ such that $S\cup\{\zeta\}$ is continuous
and
$$\mathcal L^{\zeta,\zeta}=K\mbox{ and }\mathcal L^{\zeta,\xi}=K_\xi.$$

b) Also, the following three conditions are mutually equivalent:
\begin{enumerate}
\item{} $\zeta\in S$;

\item{} $\zeta$ can be obtained as the norm limit of the sequence $(y_{t/n})^{\otimes n}$;

\item{} $\lim\limits_{n\to\infty}\skp{\zeta_t}{(y_{t/n})^{\otimes n}}=\skp{\zeta_t}{\zeta_t}$.
\end{enumerate}
\end{prp}

\begin{rmr}In \cite{PAMS08} there was considered more general limit over the filter of all partitions of
segment $[0,t]$ instead of $\lim_{n\to\infty}(y_{t/n})^{\otimes n}$. However, we do not need such general
context.
\end{rmr}

These results are used, in the same paper, to construct units starting with mappings
$$t\mapsto\sum_{j=1}^n\varkappa_j x_t^j,\qquad\varkappa_j\in\mathbb C,\quad\sum_{j=1}^n\varkappa_j=1$$
and
\begin{equation}\label{x na b}
t\mapsto x_te^{\beta t},
\end{equation}
where $x,x^1,\dots,x^n\in S$, $\beta\in\mathcal B$, proving, also, that the resulting units, denoted by
$\varkappa_1x^1\boxplus\dots\boxplus \varkappa_nx^n$ and $x^\beta$, respectively, belongs to $S$. (Obviously
the same unit $x^\beta$ is obtained if we start with mapping $t\mapsto e^{\beta t}x_t$ instead of (\ref{x na
b}), since both of them have the same generators $\L^{x,y}(b)+\beta^*b$, $\L^{x,x}(b)+\beta^*b+b\beta$.) It
was, also, noted that the product $\boxplus$ is associative unless the expression does not make sense. In
fact, it holds
$$\varkappa_1x^1\boxplus\varkappa_2x^2\boxplus\varkappa_3x^3=
(\varkappa_1+\varkappa_2)\left(\frac{\varkappa_1}{\varkappa_1+\varkappa_2}x^1\boxplus\frac{\varkappa_2}{\varkappa_1+\varkappa_2}x^2\right)
    \boxplus\varkappa_3x^3,$$
provided that $\varkappa_1+\varkappa_2\neq0$, and a similar equality
$$\varkappa_1x^1\boxplus\varkappa_2x^2\boxplus\varkappa_3x^3=\varkappa_1x^1\boxplus
(\varkappa_2+\varkappa_3)\left(\frac{\varkappa_2}{\varkappa_2+\varkappa_3}x^2\boxplus\frac{\varkappa_3}{\varkappa_2+\varkappa_3}x^3\right)
$$
provided that $\varkappa_2+\varkappa_3\neq0$.

The kernels of $x^\beta$ are given by
\begin{equation}\label{stepenovanje beta}\begin{gathered}
\mathcal
L^{x^\beta,x^\beta}=\mathcal{L}^{x,x}+\beta^{*}\mathrm{id}_{\mathcal{B}}+\mathrm{id}_{\mathcal{B}}\beta,\\
\mathcal L^{x^\beta,\xi}=\mathcal{L}^{x,{\xi}}+{\beta}^{*}\mathrm{id}_{\mathcal{B}}.
\end{gathered}\end{equation}

For our purpose, however, it is useful to substitute the complex numbers $\varkappa_j$ by elements of
$\mathcal B$. In other words we have

\begin{prp} Suppose that a continuous set $S$ of units generates a product system $E$. Let $x^j\in S$, let
let $\varkappa_j\in\mathcal B$, $j=1,\dots,n$ and let $\sum\varkappa_j=1$. Then the functions
$$t\mapsto\sum_{j=1}^n\varkappa_jx_t^j\quad\mbox{ and }t\mapsto\sum_{j=1}^nx_t^j\varkappa_j$$
satisfy all assumptions of Proposition \ref{LS}, and the resulting units belong to $S$. We shall denote them
by
$$\varkappa_1x^1\boxplus\dots\varkappa_nx^n\quad\mbox{ and }x^1\varkappa_1\boxplus\dots\boxplus
x^n\varkappa_n.$$
\end{prp}

\begin{proof} The proof is essentially the same as in \cite{PAMS08} and we shall omit it. We only give the
sketch of proof in the case $n=2$.

The infinitesimal generators of the mappings
$$t\mapsto x_t^1\varkappa_1+x_t^2\varkappa_2\quad\mbox{ and}\quad t\mapsto \varkappa_1 x_t^1+\varkappa_2 x_t^2$$
are
$$
K=\varkappa_1^*\mathcal{L}^{x^1,x^1}\varkappa_1+\varkappa_1^*\mathcal{L}^{x^1,x^2}\varkappa_2+\varkappa_2^*\mathcal{L}^{x^2,x^1}\varkappa_1
+\varkappa_2^*\mathcal{L}^{x^2,x^2}\varkappa_2,$$
$$K_{\xi}=\varkappa_1^*\mathcal{L}^{x^1,\xi}+\varkappa_2^*\mathcal{L}^{x^2,\xi}$$
and
$$
K=\mathcal{L}^{x^1,x^1} L_{\varkappa_1^*}  R_{\varkappa_1}+\mathcal{L}^{x^1,x^2} L_{\varkappa_1^*} R_{\varkappa_2}+\mathcal{L}^{x^2,x^1} L_{\varkappa_2^*} R_{\varkappa_1}+\mathcal{L}^{x^2,x^2} L_{\varkappa_2^*} R_{\varkappa_2},$$
$$K_{\xi}=\mathcal{L}^{x^1,\xi} L_{\varkappa_1^*}+\mathcal{L}^{x^2,\xi} L_{\varkappa_2^*},$$
where $L_a,R_a:\mathcal{B}\rightarrow\mathcal{B}$ are the left and right multiplication operators for
$a\in\mathcal B$.

It is not difficult to verify all assumptions stated in Proposition \ref{LS}.
\end{proof}

We, also, need the following Lemmata.

\begin{lmm}\label{lema}Let $\mathcal U$ be the set of all uniformly continuous units on a given product system $E$, and
let $x,y\in\mathcal U$. If for all $\xi\in\mathcal U$ there holds $\mathcal L^{x,\xi}=\mathcal L^{y,\xi}$
then $x=y$.
\end{lmm}

\begin{proof} From (\ref{K preko L}) we obtain
$$\skp{x_t}{b\xi_t}=\skp{y_t}{b\xi_t}.$$
For $b=1$, $\xi=x$, it becomes $\skp{x_t}{x_t}=\skp{y_t}{x_t}$, and for $b=1$, $\xi=y$ it becomes
$\skp{x_t}{y_t}=\skp{y_t}{y_t}$. Combining the last two relations we find
$\skp{x_t-y_t}{x_t-y_t}=0$.\end{proof}

\begin{rmr}\label{lemaremark}The assumption "for all $\xi\in\mathcal U$" is superfluous in the previous Lemma.
We only use the equality for $\xi=x$ and $\xi=y$.
\end{rmr}

\begin{lmm}\label{normirana jedinica}
Let $x$ be a continuous unit on some product system $E$. Then $x^{-\beta/2}$ is a unital unit, where
$\beta=\L^{x,x}(1)$. Moreover, if $x$ is central, then $x^{-\beta/2}$ is central, as well.
\end{lmm}

\begin{proof} Since it always holds
\begin{equation}\label{obrat}\L^{y,x}(b)=(\L^{x,y}(b^*))^*
\end{equation}
we get $\beta=\beta^*$. Further, by (\ref{stepenovanje beta}) we obtain
$\L^{x^{-\beta/2},x^{-\beta/2}}(1)=\L^{x,x}(1)-\beta^*/2-\beta/2=0$, and therefore
$$\skp{x_t^{-\beta/2}}{x_t^{-\beta/2}}=\left(\exp t\L^{x^{-\beta/2},x^{-\beta/2}}\right)(1)=1.$$

If, in addition, $x$ is central, then
$$be^{t\beta}=b\skp{x_t}{x_t}=\skp{x_tb^*}{x_t}=\skp{x_t}{x_tb}=e^{t\beta}b,$$
which implies $b\beta=\beta b$ for all $b\in\mathcal B$. As it is easy to see,
$x^{-\beta/2}_t=x_te^{-t\beta/2}$ and we conclude that $x^{-\beta/2}$ is, also, central.
\end{proof}

\section{Definition of index}\label{secdefind}

Let $E$ be a product system. We define the index as a quotient of a certain set of continuous units on $E$ by
a suitable inner product. Thus, the index is defined rather as operand on a set of continuous units then on a
product system. However, choosing a reference unit $\omega$, there is a maximal continuous set of units
$\mathcal U_\omega$ that contains $\omega$ (in the product system $E$) - see next Proposition. Therefore we
refer the index as $\ind(E,\omega)$ and show that it is independent on the choice of $\omega$ in the same
continuous set of units.

\begin{prp}Let $\mathcal U$ denote the set of all continuous units on some product system $E$. We define the
relation $\sim$ on $\mathcal U$ by
$$x\sim y\Leftrightarrow \{x,y\}\mbox{ is a continuous set.}$$
This relation is an equivalence relation.
\end{prp}

\begin{proof} This relation is obviously reflexive and symmetric. We have only to prove that it is transitive,
i.e.\ that $\{\xi,\eta\}$ and $\{\eta,\zeta\}$ are continuous sets implies that $\{\xi,\zeta\}$ is also a
continuous set.

It suffices to prove the uniform continuity of the mapping $b\mapsto\skp{\xi_t}{b\zeta_t}$, at $t=0$. We
begin considering the difference $\xi_t-\eta_t$. Choosing $b=1$ we have
$$\skp{\xi_t-\eta_t}{\xi_t-\eta_t}=\K^{\xi,\xi}(1)-\K^{\xi,\eta}_t(1)-\K^{\eta,\xi}_t(1)+\K^{\eta,\eta}_t(1)\to1-1-1+1=0.$$

We have, also
$$\skp{\xi_t}{b\zeta_t}=\skp{\xi_t-\eta_t}{b\zeta_t}+\skp{\eta_t}{b\zeta_t}.$$
The second summand is uniformly continuous in $b$, by continuity of $\{\eta,\zeta\}$ and, consequently, tends
to $b$ (as $t\to0+$), whereas for the first summand the following estimate holds
$$||\skp{\xi_t-\eta_t}{b\zeta_t}||\le||\xi_t-\eta_t||\;||b\zeta_t||\le C||\xi_t-\eta_t||\to0.$$

(See, also the proof of \cite[Lemma 4.4.11]{JFA04}.)\end{proof}

Thus, the set $\mathcal U$ can be decomposed into mutually disjoint collection of maximal continuous set of
units.

Let $E$ be a product system over a unital $C^*$-algebra $\mathcal B$ with at least one continuous unit. (In
view of \cite[Definition 4.4]{SK03} this means that $E$ is non type $III$ product system.) Further, let
$\omega$ be an arbitrary continuous unit in $E$ and let $\mathcal U=\mathcal U_\omega$ be the set of all
uniformly continuous units that are equivalent to $\omega$.

We define the addition and multiplication by $b\in\mathcal B$ on $\mathcal U_\omega$ by
\begin{equation}\label{operacije}
x+y=x\boxplus y\boxplus-\omega,\quad b\cdot x=bx\boxplus(1-b)\omega,\quad x\cdot b=xb\boxplus\omega(1-b).
\end{equation}

The kernels of $x+y$, $x\cdot a$, $a\cdot x$ are
\begin{equation}\label{sabiranje}\begin{gathered}
\mathcal
L^{x+y,x+y}=\mathcal{L}^{x,x}+\mathcal{L}^{x,y}-\mathcal{L}^{x,\omega}+\mathcal{L}^{y,x}+\mathcal{L}^{y,y}-\mathcal{L}^{y,\omega}-
\mathcal{L}^{\omega,x}-\mathcal{L}^{\omega,y}+\mathcal{L}^{\omega,\omega},\\
\mathcal L^{x+y,\xi}=\mathcal{L}^{x,\xi}+\mathcal{L}^{y,\xi}-\mathcal{L}^{\omega,\xi},
\end{gathered}\end{equation}
\begin{equation}\label{mnozenje}\begin{gathered}
\mathcal L^{x\cdot a,x\cdot a}=a^{*}\mathcal{L}^{x,x}a+(1-a)^{*}\mathcal{L}^{\omega,x} a+a^{*}
\mathcal{L}^{x,\omega}(1-a)+ (1-a)^{*}\mathcal{L}^{\omega,\omega}(1-a),\\
\mathcal L^{x\cdot
a,\xi}=a^{*}\mathcal{L}^{x,\xi}+(1-a)^{*}\mathcal{L}^{\omega,\xi},\ \xi \in \mathcal{U},$$
\end{gathered}\end{equation}
\begin{equation}\label{mnozenje-levo}\begin{gathered}
\mathcal L^{a\cdot x,a\cdot x}=\mathcal{L}^{x,x}L_{a^*}R_a+\mathcal{L}^{\omega,x}L_{1-a^*}R_a+
\mathcal{L}^{x,\omega}L_{a^*}R_{1-a}+\mathcal{L}^{\omega,\omega}L_{1-a^*}R_{1-a},\\
\mathcal L^{a\cdot x,\xi}=\mathcal{L}^{x,\xi}L_{a^*}+\mathcal{L}^{\omega,\xi}L_{1-a^*},\ \xi \in
\mathcal{U}.$$
\end{gathered}\end{equation}

We, also, define an equivalence relation $\approx$ by: $x\approx y$ if and only if $x=y^\beta$ for some
$\beta\in\mathcal B$.

\begin{thm}\label{modul}
a) The set $\mathcal U$ with respect to operations defined by (\ref{operacije}) is a left-right $\mathcal
B-\mathcal B$ module.

b) The relation $\approx$ is an equivalence relation and it is compatible with all algebraic operations in
$\mathcal U$, i.e.
$$x\approx y\Longrightarrow x\cdot b\approx y\cdot b,\:b\cdot x\approx b\cdot y,$$
$$x^1\approx x^2, y^1\approx y^2\Longrightarrow x^1+y^1\approx x^2+y^2;$$
\end{thm}

\begin{proof}
a) The associativity follows from the associativity of $\boxplus$. In more details, both $(x+y)+z$ and
$x+(y+z)$ are equal to $x\boxplus y\boxplus z\boxplus(-2\omega)$.

The neutral element is $\omega$ and the inverse is $2\omega\boxplus(-x)$ which can be easily checked.

Commutativity is obvious.

The other axioms of left-right $\mathcal B-\mathcal B$ module $(x\cdot a)\cdot b=x\cdot(ab)$, $a\cdot(b\cdot
x)=(ab)\cdot x$, $a\cdot(x+y)=a\cdot x+a\cdot y$, $(x+y)\cdot a=x\cdot a+y\cdot a$ and $1\cdot x=x\cdot 1=x$
can be easily checked comparing the kernels.

b) Reflexivity follows choosing $\beta=0$. If $x=y^\beta$ then $\mathcal
L^{x,\xi}=\mathcal{L}^{y,{\xi}}+{\beta}^{*}\mathrm{id}_{\mathcal{B}}$, and hence
$$
\mathcal L^{y,\xi}=\mathcal{L}^{x,{\xi}}-{\beta}^{*}\mathrm{id}_{\mathcal{B}}=\mathcal L^{x^{-\beta},\xi},
$$
from which and from Lemma \ref{lema} the symmetry follows.

Transitivity. If $x=y^\beta$ and $y=z^\alpha$ then
$$\mathcal{L}^{x,\xi}=\mathcal{L}^{y,\xi}+{\beta}^*
\mathrm{id}_{\mathcal{B}}=\mathcal{L}^{z,\xi}+(\alpha+\beta)^*\mathrm{id}_{\mathcal{B}}=\mathcal{L}^{z^{\alpha+\beta},\xi}$$
for any $\xi\in\mathcal{U}$. From this and from Lemma \ref{lema} we conclude $x=z^{\alpha+\beta}$.

Let us now prove that the result of addition and multiplication by $b\in\mathcal B$ does not depend on the
choice of $\beta$. Indeed, if $x, y \in \mathcal{U}$ and $x_1=x^\beta$, $y_1=y^\alpha$, for some $\alpha$,
$\beta \in \mathcal{B}$. Then, by (\ref{sabiranje}) and (\ref{stepenovanje beta})
$$\begin{aligned}\mathcal{L}^{x_1+y_1,\xi}=&\L^{x_1,\xi}+\L^{y_1,\xi}-\L^{\omega,\xi}=\L^{x,\xi}+\beta^*\id_{\mathcal
B}+\L^{y_1,\xi}+\alpha^*\id_{\mathcal B}-\L^{\omega,\xi}=\\
=&\L^{x+y,\xi}+(\alpha+\beta)^*\id_{\mathcal B}=\L^{(x+y)^{\alpha+\beta},\xi}
\end{aligned}$$
for any $\xi \in \mathcal{U}$. It follows that, again using Lemma \ref{lema}
$$x_1+y_1=(x+y)^{\alpha+\beta}.$$

Further, let $x_1=x^\beta$ and let $a \in \mathcal{B}$. Then by (\ref{mnozenje})
$$\begin{aligned}
\mathcal L^{x_1\cdot
a,\xi}=&a^{*}\mathcal{L}^{x_1,\xi}+(1-a)^{*}\mathcal{L}^{\omega,\xi}=a^*(\L^{x,\xi}+\beta^*\id_{\mathcal
B})+(1-a)^*\L^{\omega,\xi}=\\
=&\L^{x\cdot a,\xi}+a^*\beta^*\id_{\mathcal B}=\mathcal{L}^{(x \cdot a)^{\beta a},\xi},
\end{aligned}$$
for any $\xi\in\mathcal{U}$. It follows, once again using Lemma \ref{lema}, that
$$x_1 \cdot a=(x \cdot a)^{\beta a}.$$

A similar argument shows that $a\cdot x_1=(a\cdot x)^{a\beta}$.
\end{proof}

We can immediately form the quotient module $\mathcal U/\approx$. However, it might not be the accurate
choice, taking into account possible choices of inner product. Thus, we are looking for the most suitable
choice of a $\mathcal B$ valued inner product on $\mathcal U$. For a while, we shall consider a family of
candidates. Namely, for every positive element $b \in \mathcal{B}$ there is a map $\langle\ ,\
\rangle_b:\mathcal{U} \times \mathcal{U} \longrightarrow \mathcal{B}$ given by
\begin{equation}\label{skpdef}\langle
x,y\rangle_{b}=(\mathcal{L}^{x,y}-\mathcal{L}^{x,\omega}-\mathcal{L}^{\omega,y}+\mathcal{L}^{\omega,\omega})(b),
\end{equation}
where $\omega$ is the same as in (\ref{operacije}).

Any of these mappings is $\mathcal B$-valued semi-inner product (in the sense that it can be degenerate,
i.e.~$\skp xx_b=0$ need not imply $x=0$). Nevertheless, it satisfies all other customary properties.

\begin{prp}\label{skp} The pairing (\ref{skpdef}) satisfies the following
properties:

\begin{enumerate}
\renewcommand{\theenumi}{\alph{enumi}}

\item{}\label{skp-i} For all $x,y,z \in \mathcal{U}$, and $\alpha,\beta \in \mathbb{C}$ $\langle x,\alpha y+\beta
z \rangle_{b}=\alpha \langle x,y \rangle_{b}+\beta \langle x,z \rangle_{b}$;

\item{}\label{skp-ii} For all $x,y \in \mathcal{U}$, $a \in \mathcal{B}$ $\langle x,y\cdot a \rangle_{b}=\langle x,y
\rangle_{b}a$;

\item{}\label{skp-iii} For all $x,y \in \mathcal{U}$\ $\langle x,y \rangle_{b}=\langle y,x\rangle_{b}^{*}$;

\item{}\label{skp-iv} For all $x \in \mathcal{U}$\ $\langle x,x\rangle_{b}\geq 0$;

\item{}\label{skp-v} If $x\approx x'$ and $y\approx y'$ then $\skp xy_b=\skp{x'}{y'}_b$;

\item{}\label{skp-vi} For all $x,y\in\mathcal U$, $0\le a\in\mathcal B$ $\skp x{a\cdot y}_1=\skp xy_a$;

\item{}\label{skp-vii} If $0\le b(\in\mathcal B)\le1$ then for all $x\in\mathcal U$ we have $\skp xx_b\le\skp xx_1$.

\item{}\label{skp-viii} There holds $\displaystyle\skp{x-y}{x-y}_1=\lim_{t\to0+}\frac{\skp{x_t-y_t}{x_t-y_t}}t$.

\end{enumerate}
\end{prp}

\begin{proof} (\ref{skp-i})-(\ref{skp-iii}) is easy to check. (\ref{skp-iv}) follows, since $\mathcal L$ is conditionally CPD,
precisely we can put $n=2$, $x^1=x$, $x^2=\omega$, $a_1=a_2=\sqrt b$, $b_1=1$, $b_2=-1$ in (\ref{CCPD}).
(\ref{skp-v}) follows from the cancellation of terms $\beta^*b$ and $b\gamma$ in expanded form of
$\skp{x'}{y'}$, where $x'=x^\beta$ and $y'=y^\gamma$.

To conclude (\ref{skp-vi}), expand $\skp x{a\cdot y}$ and use (\ref{mnozenje-levo}).

(\ref{skp-vii}) - Since $\mathcal{L}$ is conditionally completely positive definite, we get
$\mathcal{L}^{x,x}(1-b)-\mathcal{L}^{x,\omega}(1-b)-\mathcal{L}^{\omega,x}(1-b)+
\mathcal{L}^{\omega,\omega}(1-b)\geq0$. It follows that $\langle x,x \rangle_b \leq \langle x,x \rangle_{1}.$

(\ref{skp-viii}) follows from (\ref{skpdef}) and definition of $\L$ after few cancellations.
\end{proof}

Our choice for the inner product will be $\skp\cdot\cdot_1$, which we shall abbreviate to $\skp\cdot\cdot$ in
further, i.e.\ if we omit the index $b$ we shall assume $b=1$. From the properties (\ref{skp-ii}) -
(\ref{skp-iv}) of the previous Proposition we can derive the Cauchy Schwartz inequality (see
\cite[Proposition 1.1]{Lance} or \cite[Proposition 1.2.4]{MT})
\begin{equation}\label{CauchySchwartz}
\skp xy\skp xy^*\le\skp xx||\skp yy||.
\end{equation}
It follows that the set $N=\{x\in\mathcal U\:|\:\skp xx=0\}$ is equal to $\{x\in\mathcal U\:|\:\forall
y\in\mathcal U,\;\skp xy=0\}$ and that it contains $\{\omega^\beta\:|\:\beta\in\mathcal B\}$ (property
(\ref{skp-v})). From this and from (\ref{skp-i}) - (\ref{skp-iii}) and (\ref{skp-vi}) we conclude that $N$ is
a submodule of $\mathcal U$, and $\mathcal U/N$ is a pre-Hilbert left right $\mathcal B$ - $\mathcal B$
module.

\begin{dfn}\label{definicija indeksa}
Let $E$ be a product system, and let $\omega$ be a continuous unit on $E$. The index of a pair $(E,\omega)$
is the completion of pre-Hilbert left-right module $\mathcal U/\sim$, where $\mathcal U=\mathcal U_\omega$ is
the maximal continuous set of units containing $\omega$, and $\sim$ is the equivalence relation defined by
$x\sim y$ if and only if $x-y\in N$. Naturally, the index will be denoted by $\ind(E,\omega)$.
\end{dfn}

\begin{rmr}If $E$ can be embedded into a spatial product system, as we shall see in the next section, the
completion is unnecessary.
\end{rmr}

\begin{rmr}
If $\{\omega,\omega'\}$ is a continuous set, then $\ind(E,\omega)\cong\ind(E,\omega')$. Indeed, then
$\mathcal U_\omega=\mathcal U_{\omega'}$ and the isometric isomorphism is given by translation $x\mapsto
x\boxplus-\omega\boxplus\omega'$
\end{rmr}

The following definition of {\em continuous} product system \cite[Section 7]{SK03} will allow us to speak of
the index of $E$ without highlighting the unit $\omega$.

\begin{dfn}Continuous product system is a product system $(E_t)_{t\ge0}$, together with a family of isometric
embeddings $i_t: E_t\to E$ into a unital Hilbert bimodule, which satisfies

\begin{enumerate}
\item{} For every $y_s\in E_s$ there exists a continuous section $(x_t)\in CS_i(E)$ such that
$y_s=x_s$;

\item{} For every pair $x,y\in CS_i(E)$ of continuous sections the function $(s,t)\mapsto i_{s+t}(x_s\otimes
y_t)$ is continuous;
\end{enumerate}
where the set of continuous sections (with respect to $i$) is
$$CS_i(E)=\{x=(x_t)_{t\ge0}\:|\:x_t\in E_t,\:t\mapsto i_tx_t\mbox{ is continuous}\}.$$
\end{dfn}

By \cite[Theorems 7.5 and 7.7]{SK03} (see also \cite[Theorems 2.4 and 2.5]{PAMS10}) there is at most one
continuous structure on $E$ that makes a given continuous unit $\omega$ a continuous section. Further, given
a continuous unit $\omega\in CS_i(E)$, the set $\mathcal U_\omega$ coincides with the set of all continuous
units that belongs to $CS_i(E)$. Indeed, if continuous unit $x$ belongs to $CS_i(E)$ then $\omega\sim x$ by
\cite[Theorem 7.7]{SK03}. Conversely, if $\omega\sim x$ and $x$ is a continuous unit then
$$||x_{t+\e}-x_t||\le||x_t\otimes
x_\e-x_t\otimes\omega_\e||+||x_t\otimes\omega_\e-x_t||\le||x_t||(||x_\e-\omega_\e||+||\omega_\e-1||)\to0,$$
as $\e\to0+$, because the first summand tends to zero by $\omega\sim x$, whereas the second summand tends to
zero, by $\omega\in CS_i(E)$. Left continuity follows from right, since
$||x_{t-\e}-x_t||\le||x_{t-\e}||\,||1-x_\e||$

Thus, for continuous product systems, we shall not underlain the unit $\omega$, i.e.\ we shall write
$\ind(E)$.

The class of all continuous product systems is a category, if morphism are defined as follows.

\begin{dfn}\label{Contstr}
The mapping $\theta:E\to F$ between two continuous product systems (with embeddings $i$ and $j$
respectively) is morphism if:

\begin{enumerate}
\item{}\label{Contstr-i} $\theta|_{E_t}$ is a bounded adjointable $\mathcal B-\mathcal B$ linear mappings
$\theta_t:E_t \rightarrow F_t$ fulfilling $\theta_{t+s}=\theta_t\otimes\theta_s$ and
$\theta_0={\id}_{\mathcal B}$;

\item{}\label{Contstr-ii} Both $\theta$ and $\theta^*$ preserve continuous structure, i.e.\ if
$(x_t)_{t\ge0}\in CS_i(E)$ is a continuous section then $(\theta(x_t))_{t\ge0}\in CS_j(F)$, and if
$(y_t)_{t\ge0}\in CS_j(F)$ then $(\theta^*(y_t))_{t\ge0}\in CS_i(E)$;

\item{}\label{Contstr-iii} $\displaystyle\limsup_{t\to0+}||\theta_t||<+\infty$.
\end{enumerate}
\end{dfn}

\begin{rmr} In \cite[Section 2]{KPR06} morphisms are defined as mappings that satisfies only condition
(\ref{Contstr-i})
(in previous Definition). This definition is, however, pure algebraic, and we can not say anything about
continuous structure, without additional assumptions.
\end{rmr}

\begin{prp} The index is a covariant functor from the category of continuous product systems over $\mathcal B$ to the
category of all left-right $\mathcal B-\mathcal B$ modules.
\end{prp}

\begin{proof}
Let $\theta:E\to F$ be a morphism. For a reference unit in $\mathcal U_F$ choose $\omega'=\theta(\omega)$.
For an arbitrary unit $x=(x_t)$ on $E$, $\theta(x)=(\theta_t(x_t))$ is a unit on $F$ since
$\theta_{t+s}(x_{t+s})=\theta_{t+s}(x_t\otimes x_s)=\theta_t(x_t)\otimes\theta_s(x_s)$ and $\theta_0(x_0)=1$.
Further, if $x$ is continuous, then $x\in CS_i(E)$, implying $\theta(x)\in CS_j(F)$, that is $\theta(x)$ is
continuous.

Using Lemma \ref{lema} and noting that $\L^{\theta(x+y),\xi}=\L^{\theta(x)+\theta(y),\xi}$, it follows that
$\theta(x+y)=\theta(x)+\theta(y)$ for $x,y\in {\mathcal U_{E}}$. Similarly, $\theta(x\cdot a)=\theta(x)\cdot
a$ and $\theta(a\cdot x)=a\cdot \theta(x)$, $a\in{\mathcal B}$.

Hence, the mapping $\mathcal U_E\ni x\mapsto\theta(x)\in{\mathcal U}_F$ is an algebraic homomorphism.

Let $\theta^*:F\to E$ denote the morphism which fibers are $\theta_t^*:F_t\to E_t$, and let $x=(x_t)$ be a
unit in $E$. Then $(\theta_t^*\theta_tx_t)$ is also a unit. Let $y\sim y_1$ in $E$, and denote
$\psi_t=\theta_t^*\theta_t$. Then, for all $x\in\mathcal U_E$, using Proposition \ref{skp} (\ref{skp-viii}),
we obtain $\skp{\theta x}{\theta y}=\skp{\psi x-\psi\omega}y$ and hence $\skp{\theta x}{\theta y-\theta
y_1}=\skp{\psi x-\psi\omega}{y-y_1}=0$, implying $\theta y\sim\theta y_1$. Thus, we obtain a well defined
homomorphism $\mathcal U_E/\sim\ni[x]\mapsto\ind(\theta)([x])=[\theta x]\in\mathcal U_F/\sim$.

Let us prove that $\ind(\theta)$ is an adjointable mapping. For any $y\in\mathcal U_F$, $\theta^*y\in\mathcal
U_E$. Then we have
$$\begin{aligned}\skp{\ind(\theta)x}y=&\L^{\theta x,y}(1)-\L^{\theta
x,\omega'}(1)-\L^{\omega',y}(1)+\L^{\omega',\omega'}(1)=\\
=&\L^{x,\theta^*y}(1)-\L^{x,\theta^*\omega'}-\L^{\omega,\theta^*y}(1)+\L^{\omega,\theta^*\omega'}(1)=\skp
x{\theta^*y-\theta^*\omega'}.
\end{aligned}$$
This shows that the adjoint of $\ind(\theta)$ is the mapping $(\ind(\theta))^*(y)=\theta^*y-\theta^*\omega$
(the composition of $\ind(\theta^*)$ and translation $x\mapsto x-\theta^*\omega'$.

Finally, let us prove that $\ind(\theta)$ is bounded, and, therefore, that it can be extended to $\ind(E)$.
Using Proposition \ref{skp} (\ref{skp-viii}) and \cite[Proposition1.2]{Lance} (or \cite[Corollary 2.1.6]{MT})
we obtain

\begin{multline*}
\skp{\theta x}{\theta
x}=\lim_{t\to0+}\frac{\skp{\theta_tx_t-\theta_t\omega_t}{\theta_tx_t-\theta_t\omega_t}}t\le\\
\le\limsup_{t\to0+}||\theta_t||^2\frac{\skp{x_t-\omega_t}{x_t-\omega_t}}t\le(\limsup_{t\to0+}||\theta_t||^2)\skp
xx.
\end{multline*}

Hence $\ind(E)\in B^{a,bil}(\ind(E);\ind(F))$ is a morphism in the category of all left-right $\mathcal
B-\mathcal B$ modules over $\mathcal B$.

It can be easily seen  that ${\ind}({\id}_{E})={\id}_{{\ind}(E)}$ and ${\ind}(\psi \theta)={\ind}(\psi)
{\ind}(\theta)$ for all morphisms $\theta$ between product systems $E$ and $F$ and $\psi$ between product
systems $F$ and $G$.
\end{proof}

\begin{rmr}
The induced mapping $\ind(\theta)$ preserves the relation $\approx$. Indeed, if $x'=x^{\beta}$, $\beta\in
{\mathcal B}$, then $\L^{\theta(x'),\xi}=\L^{\theta(x)^{\beta},\xi}$ for $\xi \in \mathcal U_{F}$ and Lemma
\ref{lema} implies that $\theta(x')=\theta(x)^{\beta}$.
\end{rmr}

\begin{rmr}If the condition (\ref{Contstr-iii}) is suppressed, we only can obtain that $\ind(\theta)$ is
densely defined adjointable (possibly unbounded) operator from $\ind(E)$ to $\ind(F)$.
\end{rmr}

\begin{cor}If $E$ and $F$ are algebraically isomorphic product systems then their indices coincide.
\end{cor}

\begin{proof}$E$ and $F$ are algebraically isomorphic if there is a mapping $\theta:E\to F$, whose fibers are
unitary operators. The norm of unitary operators is equal to $1$, and the condition (\ref{Contstr-iii}) in
Definition \ref{Contstr} is fulfilled. Further, $\skp{\theta_tx_t}{b\theta_ty_t}=\skp{x_t}{by_t}$, from which
we conclude that $\theta$ converts continuous units into continuous, as well as $\skp{\theta x}{\theta
x}=\skp xx$. Therefore, in this case $\ind(\theta)$ is unitary operator, implying $\ind(E)\cong\ind(F)$, or
more precisely $\ind(E,\omega)=\ind(F,\theta(\omega))$.
\end{proof}

\section{Subadditivity of the index}\label{sectensor}

Given two product systems, $E$ over a unital $C^*$-algebra $\mathcal A$ and $F$ over $\mathcal B$, we can
consider its (outer) tensor product $E\otimes F$ as a product system over $\mathcal A\otimes \mathcal B$,
taking pointwise outer tensor product $E_t\otimes F_t$ as a Hilbert module over $\mathcal A\otimes \mathcal
B$. (Here $\mathcal A\otimes\mathcal B$ denotes the spatial tensor product of $C^*$-algebras.) This is the
direct generalization of tensor product within the category of Arveson product system. On the other hand, it
appears as a product system generated by the action of $E_0$ semigroup $\alpha\otimes\beta$ on $\mathcal
M\otimes\mathcal N$, where $\alpha$ and $\beta$ are $E_0$ semigroups on type $II_1$ factors $\mathcal M$ and
$\mathcal N$ (see \cite{JOT04}).

It is easy to see that units $x_t$ on $E$, and $y_t$ on $F$ gives rise to the unit $x_t\otimes y_t$ on
$E\otimes F$. The corresponding semigroup is (evaluated on elementary tensors)
$$\skp{x_t\otimes x_t'}{(a\otimes b)(y_t\otimes y_t')}=\skp{x_t}{ay_t}\otimes\skp{x'_t}{by_t'},$$
and its continuity is obvious. Thus, we have the mapping $\mathcal U_E\times\mathcal U_F\to\mathcal
U_{E\otimes F}$. If $\omega$ and $\omega'$ are reference units in $\ind E$ and $\ind F$, then it is natural
to choose $\omega\otimes\omega'$ to be the reference unit in $E\otimes F$.

First, we list some basic properties of $x\otimes y$.

\begin{prp}\label{racunica sa otimes}
Let $1$ and $1'$ denote the identity elements in $\mathcal A$ and $\mathcal B$. Then for all $a\in\mathcal
A$, $b\in\mathcal B$, $x,y\in\mathcal U_E$ and $x',y'\in\mathcal U_F$ there holds:

\begin{enumerate}
\renewcommand{\theenumi}{\alph{enumi}}

\item{}\label{rac-a} $\L^{x\otimes x',y\otimes y'}(a\otimes b)=a\otimes\L^{x',y'}(b)+\L^{x,y}(a)\otimes b$ - Leibnitz rule;

\item{}\label{rac-b} $\skp{x\otimes x'}{y\otimes y'}=1\otimes\skp{x'}{y'}+\skp xy\otimes 1'$, where the inner products are those
in $\mathcal U_{E\otimes F}$, $\mathcal U_E$ and $\mathcal U_F$, respectively;

\item{}\label{rac-c} $(x\otimes\omega')\cdot(\alpha\otimes 1')=(x\cdot\alpha)\otimes\omega'$, $(\omega\otimes
y)\cdot(1\otimes\beta)=\omega\otimes(y\cdot\beta)$, where $\cdot$ denotes the multiplying in modules
$\mathcal U_{E\otimes F}$, $\mathcal U_E$ and $\mathcal U_F$, respectively;

\item{}\label{rac-d} $x\otimes y=x\otimes\omega'+\omega\otimes y$, where addition is those in module $\mathcal U_{E\otimes F}$;

\item{}\label{rac-e} $(x\otimes\omega')\cdot(1\otimes\beta)=(1\otimes\beta)\cdot(x\otimes\omega')$ and $(\omega\otimes
y)\cdot(\alpha\otimes 1')=(\alpha\otimes 1')\cdot(\omega\otimes y)$;

\item{}\label{rac-f} $\skp{x\otimes\omega'}{\omega\otimes y}=0$.
\end{enumerate}
\end{prp}

\begin{proof}(\ref{rac-a}) Straightforward calculation;

(\ref{rac-b}) Follows from (\ref{rac-a}) and definition of the inner product;

(\ref{rac-c}) Using part (\ref{rac-a}), (\ref{mnozenje}) and (\ref{mnozenje-levo}), after straightforward,
but unpleasant calculations we conclude that all kernels:
$$\L^{(x\otimes\omega')\cdot(\alpha\otimes1'),(x\otimes\omega')\cdot(\alpha\otimes1')}(a\otimes b)\qquad
\L^{(x\otimes\omega')\cdot(\alpha\otimes1'),x\cdot\alpha\otimes\omega'}(a\otimes b)$$
$$\L^{x\cdot\alpha\otimes\omega',(x\otimes\omega')\cdot(\alpha\otimes1')}(a\otimes b)\qquad
\L^{x\cdot\alpha\otimes\omega',x\cdot\alpha\otimes\omega'}(a\otimes b)$$
are equal to
\begin{multline*}a\otimes\L^{\omega',\omega'}(b)+(\alpha^*\L^{x,x}(a)\alpha+\alpha^*\L^{x,\omega}(a)(1-\alpha)+\\
    +(1-\alpha^*)\L^{\omega,x}(a)\alpha+(1-\alpha^*)\L^{\omega,\omega}(a)(1-\alpha))\otimes b.
\end{multline*}
By this, Lemma \ref{lema} and Remark \ref{lemaremark} we conclude the first equality. The second follows
similarly.

(\ref{rac-d}) After few steps we get
\begin{multline*}\L^{x\otimes\omega'+\omega\otimes y,x\otimes\omega'+\omega\otimes y}(a\otimes b)=
\L^{x\otimes\omega'+\omega\otimes y,x\otimes y}(a\otimes b)=\L^{x\otimes y,x\otimes\omega'+\omega\otimes
y}(a\otimes b)=\\=\L^{x\otimes y,x\otimes y}(a\otimes b)=a\otimes\L^{y,y}(b)+ \L^{x,x}(a)\otimes b;
\end{multline*}

(\ref{rac-e}) Once again, using part (\ref{rac-a}), (\ref{mnozenje}) and (\ref{mnozenje-levo}) we conclude
that all kernels
$$\L^{(x\otimes\omega')\cdot(1\otimes\beta),(x\otimes\omega')\cdot(1\otimes\beta)}(a\otimes b)\qquad
\L^{(x\otimes\omega')\cdot(1\otimes\beta),(1\otimes\beta)\cdot(x\otimes\omega')}(a\otimes b)$$
$$\L^{(1\otimes\beta)\cdot(x\otimes\omega'),(x\otimes\omega')\cdot(1\otimes\beta)}(a\otimes b)\qquad
\L^{(1\otimes\beta)\cdot(x\otimes\omega'),(1\otimes\beta)\cdot(x\otimes\omega')}(a\otimes b)$$
are equal to
\begin{multline*}a\otimes\L^{\omega',\omega'}(b)+\L^{x,x}(a)\otimes\beta^*b\beta+\L^{\omega,x}(a)\otimes(1'-\beta^*)b\beta+\\
\L^{x,\omega}(a)\otimes\beta^*b(1'-\beta)+\L^{\omega,\omega}(a)\otimes(1'-\beta^*)b(1-\beta);
\end{multline*}

(\ref{rac-f}) Follows easily from (\ref{rac-b}).
\end{proof}

\begin{rmr}Note that, in general, $(x\otimes y)\cdot(\alpha\otimes\beta)\neq(x\cdot\alpha)\otimes(y\cdot\beta)$, so that
$\ind(E\otimes F)$ can not be considered as a tensor product of $\ind E$ and $\ind F$.
\end{rmr}

\begin{prp}The mapping $T:(\ind(E)\otimes\mathcal B)\oplus(\mathcal A\otimes\ind(F))\to\ind(E\otimes F)$ defined
on the elementary tensors from the dense subset $\left((\mathcal U_E/\sim)\otimes\mathcal
B\right)\oplus\break\left(\mathcal A\otimes(\mathcal U_F/\sim)\right)$ by
$$T([x]\otimes\beta,\alpha\otimes[y])=[(x\otimes\omega')\cdot(1\otimes\beta)+(\alpha\otimes
1')\cdot(\omega\otimes y)]$$
is a module homomorphism and isometric embedding.
\end{prp}

\begin{proof}First, taking into account Proposition \ref{racunica sa otimes} (parts (\ref{rac-f}) and (\ref{rac-b})), for
$z=(x\otimes\omega')\cdot(1\otimes\beta)+(\alpha\otimes 1')\cdot(\omega\otimes y)$ we obtain
\begin{equation}\label{outerT}
\skp zz=\skp xx\otimes\beta^*\beta+\alpha^*\alpha\otimes\skp yy=\skp{(x\otimes\beta,\alpha\otimes
y)}{(x\otimes\beta,\alpha\otimes y)}.
\end{equation}
Hence, we get that $T$ is well defined. Indeed, if $x\sim x_1$ then
$(x\otimes\omega')\cdot(1\otimes\beta)\sim(x_1\otimes\omega')\cdot(1\otimes\beta)$, since their difference
multiplied by itself is equal to zero. Similarly for $y\sim y_1$.

Additivity is obvious.

For right multiplication, using Proposition \ref{racunica sa otimes} parts (\ref{rac-c}) and (\ref{rac-e}),
we get
$$\begin{aligned}T(([x]\otimes\beta,\alpha\otimes[y])\cdot(a\otimes b))=&T([x\cdot a]\otimes\beta b,\alpha
a\otimes[y\cdot b])=\\
=&[(x\otimes\omega')\cdot(a\otimes\beta b)+(\omega\otimes y)\cdot(\alpha a\otimes b)]=\\
=&T([x]\otimes\beta,\alpha\otimes[y])\cdot(a\otimes b),
\end{aligned}$$
and similarly for left multiplication.

Finally, from (\ref{outerT}) it follows that $T$ is an isometry, and hence embedding.
\end{proof}

\begin{rmr}In Arveson case, i.e.\ in the case $\mathcal A=\mathcal B=\mathrm C$, the above embedding is
actually an isomorphism, due to \cite[Theorem 3.7.2 and Corollary 3.7.3]{Arv03} which asserts that any unit
$w$ in $E\otimes F$ has the form $w=u\otimes v$ for some units $u$ in $E$, and $v$ in $F$. However, almost
every substantial step in the proof of these statements fails in general situation. Therefore, it should find
either entirely different proof, or a suitable counterexample.
\end{rmr}

\section{(Sub)spatial product systems}\label{secspatial}

In this section we prove that the index of a spatial or a subspatial product system can be described more
precisely. In more details, the relations $\sim$ and $\approx$ coincide, $\mathcal U$ can be recovered from
$\ind(E)$ as $\mathcal U\cong\mathcal B\oplus\ind(E)$, and finally, completion in the Definition
\ref{definicija indeksa} is not necessary. Some of these properties can be obtained using the fact that any
spatial product system contains a subsystem isomorphic to a time ordered Fock module \cite[Theorem
6.3]{KPR06}. However, the proofs presented here are independent of this characterization, and henceforth they
don't use Kolmogorov decomposition of completely positive definite kernels.

We begin with the definition of spatial \cite[Section 2]{KPR06} and subspatial product system.

\begin{dfn}The spatial product system is a product system that contains a central unital unit. The system is
subspatial if it can be embedded into a spatial one.
\end{dfn}

\begin{rmr}Recall that unit $\omega$ is central if and only if $b\omega_t=\omega_tb$ for all $b\in\mathcal B$
and all $t\ge0$. Such a unit might not exist (see \cite[Example 4.2.4]{JFA04}). However, its well behaviour
allows to obtain plenty of interesting results.

In view of Lemma \ref{normirana jedinica}, it is enough to assume that $E$ admits a central continuous unit,
instead of assuming that it admits a central unital unit.

Note, also, that subspatial system might not be spatial (see \cite[Section 3]{PAMS10}). The converse is
trivially satisfied.
\end{rmr}

Throughout this section, the reference unit $\omega$ is always assumed to be central and we shall assume that
it is specified, even though it is not emphasized.

The following Lemma establishes the most important property of central units. Although it is very simple and
seen in many papers, we give its proof for the convenience of the reader.

\begin{lmm}If a unit $\omega$ is central then for all $b\in\mathcal B$ and all $x\in\mathcal U$
\begin{equation}\label{svojstva centralne}
\L^{x,\omega}(b)=\L^{x,\omega}(1)b.
\end{equation}

Then, also $\L^{\omega,x}(b)=b\L^{\omega,x}(1)$.
\end{lmm}

\begin{proof}If $\omega$ is central, we have
$$\L^{x,\omega}(b)=\lim_{t\to0+}\frac{\skp{x_t}{b\omega_t}-b}t=\lim_{t\to0+}\frac{\skp{x_t}{\omega_t}-1}tb=\L^{x,\omega}(1)b.$$
\end{proof}

The next Proposition allows us to translate the statements proved for spatial product systems to subspatial.

\begin{prp}Let $E$ be a subspatial product system embedded into a spatial system $\hat E$ with a central
unit $\hat\omega$, and let $\omega$ be an arbitrary unit on $E$. Then the mapping
$$\Phi:\mathcal U_E\to\{x-\omega\:|\:x\in\mathcal U_E\}\subseteq\mathcal U_{\hat E},\quad\Phi(x)=x-\omega,$$
is an embedding, where substraction is that in $\mathcal U_{\hat E}$, i.e.\
$\Phi(x)=x-\omega=x\boxplus\hat\omega\boxplus(-\omega)$.

In other words, $\mathcal U_E$ is an affine subspace of $\mathcal U_{\hat E}$.
\end{prp}

\begin{proof}Indeed,
\begin{multline*}\Phi(x+y)=\Phi(x\boxplus y\boxplus(-\omega))=(x\boxplus
y\boxplus(-\omega)\boxplus\hat\omega\boxplus(-\omega)=\\=(x\boxplus\hat\omega\boxplus(-\omega))\boxplus
(y\boxplus\hat\omega\boxplus(-\omega))\boxplus(-\hat\omega)=\Phi(x)+\Phi(y)
\end{multline*}
and also
\begin{multline*}\Phi(x\cdot
a)=\Phi(xa\boxplus\omega(1-a))=(xa\boxplus\omega(1-a))\boxplus\hat\omega\boxplus(-\omega)=\\
(x\boxplus\hat\omega\boxplus(-\omega))a\boxplus\hat\omega(1-a)=\Phi(x)\cdot a
\end{multline*}
and similarly for $\Phi(a\cdot x)$. Finally, we easily find that $\skp{\Phi(x)}{\Phi(y)}=\skp xy$.
\end{proof}

The following Proposition establishes that the relations $\approx$ and $\sim$ coincide.

\begin{prp}\label{sim-approx}
Let $E$ be a subspatial product system. Then the equivalence relation $\sim$ from Definition
\ref{definicija indeksa} is characterized as follows:
$$x\sim y\Longleftrightarrow x=y^\beta,\quad\mbox{for some }\beta\in\mathcal B.$$
\end{prp}

\begin{proof} Let us, first, assume that $E$ is spatial, and that $\omega$ is its central unit. Since both
relations $\approx$ and $\sim$ are compatible with algebraic operations and $x+\omega^\beta=x^\beta$, it
suffices to prove that $\skp xx=0$ implies $x\approx\omega$, i.e. $x^\beta=\omega$ for some $\beta\in
\mathcal B$.

Let $\skp xx_1=0$, let $b \in \mathcal{B}$,\ $b\geq0$ and let denote $\tilde{b}=b/\|b\|$. The element
$\tilde{b}$ is positive and $1-\tilde{b}\ge 0$. By Proposition \ref{skp} (\ref{skp-vii}) we have
\begin{equation}\label{b je manje}
\skp xx_{\tilde b}\le\skp xx,
\end{equation}
and hence $\langle x,x \rangle_{\tilde{b}}=0$.

From Cauchy Schwartz inequality (\ref{CauchySchwartz}), we have $\langle x,y \rangle_{\tilde{b}}=0$ for all
$y \in \mathcal{U}$. Let $\beta=\mathcal{L}^{\omega,\omega}(1)-\mathcal{L}^{\omega,x}(1) \in \mathcal{B}$.
Then $\mathcal{L}^{x^{\beta},y}(\tilde{b})=\mathcal L^{x,y}(\tilde b)+\beta^*\tilde
b=\mathcal{L}^{\omega,y}(\tilde{b})$ and hence $\mathcal{L}^{x^{\beta},y}(b)=\mathcal{L}^{\omega,y}(b)$.
Since every element in $\mathcal{B}$ is a linear combination of at most four positive elements, we get
$\mathcal{L}^{x^{\beta},y}=\mathcal{L}^{\omega,y}$. Using Lemma \ref{lema}, we conclude $x^\beta=\omega$.

Let, now, $E$ be a subspatial product system. Then it can be embedded in some spatial system $\hat E$ that
contains a central unit $\hat\omega$. If $x\sim y$, then, obviously, $x-\omega\sim y-\omega$, and by previous
part, $x-\omega=(y-\omega)^\beta$ for some $\beta\in\mathcal B$, which is equal to $y^\beta-\omega$ by
Theorem \ref{modul}. Hence $x=y^\beta$.
\end{proof}

\begin{thm}\label{kompletnost}
If $E$ is subspatial product system, then $\mathcal{U}/_{\sim}$ is a Hilbert left-right $\mathcal B$ -
$\mathcal{B}$-module.
\end{thm}

\begin{proof}
We, only, have to prove that $\mathcal U/\sim$ is norm complete. First, assume that $E$ is spatial.

Let $([x^n])$ be a Cauchy sequence in $\mathcal{U}/_{\sim}$, that is: for all $0<\varepsilon\leq1$ there is
$n_0 \in \mathbb{N}$ such that
\begin{equation}\label{Cauchy1}
\| \langle x^n-x^m, x^n-x^m \rangle_1 \|=\|[x^n]-[x^m]\|^2 < \varepsilon^2\quad\mbox{for }m,n \geq n_0.
\end{equation}
We would like to show that $([x^n])$ is convergent. This sequence is, of course, bounded. First, we prove
that the sequence $\skp{x^n}{x^n}_{\tilde b}$ is also a Cauchy sequence, where $\tilde{b}=b/\|b\|$, and
$b\geq0$ is arbitrary. For $m,n \geq n_0$, we have by (\ref{b je manje}),
\begin{multline}\label{velika}\| \langle x^n,x^n \rangle_{\tilde{b}}- \langle x^m,x^m \rangle_{\tilde{b}} \| \leq \\\| \langle
x^n-x^m, x^n-x^m \rangle_{\tilde{b}}\|+\| \langle x^n-x^m,x^m \rangle_{\tilde{b}}\|+\| \langle x^m,x^n-x^m
\rangle_{\tilde{b}}\| \leq\\
\leq \| \langle x^n-x^m, x^n-x^m \rangle_{1}\|+2\sqrt{\| \langle x^n-x^m, x^n-x^m \rangle_{1}\|}\sqrt{\|
\langle x^m,x^m \rangle_1 \|} <\\
< \varepsilon^2+2\varepsilon \sqrt{\| \langle x^m,x^m \rangle_1 \|} < \varepsilon\mathrm{\ const}.
\end{multline}

The unit $x^n$ is an arbitrary representative of the class $[x^n]$, and now we are going to pick the most
suitable one. Let $\beta_n=-\mathcal{L}^{\omega,x^n}(1)\in \mathcal{B}$, and let $x^{\beta_n}$ denote the
unit $(x^n)^{\beta_n}$. By (\ref{svojstva centralne}) we have $\L^{\xi,\omega}(b)=\L^{\xi,\omega}(1)b$. This
ensures that
$$\mathcal{L}^{\omega,x^{\beta_n}}(\tilde{b})=(\mathcal{L}^{x^{\beta_n},\omega}(\tilde{b}))^*=(\L^{x^n,\omega}(1)+(\beta^n)^*)\tilde b=0,$$
for $n\in\mathbb{N}$. By (\ref{b je manje}), we obtain
$$\mathcal{L}^{x^{\beta_n},x^{\beta_n}}(\tilde{b})-\mathcal{L}^{x^{\beta_m},x^{\beta_m}}(\tilde{b})=
\langle x^{\beta_n},x^{\beta_n} \rangle_{\tilde{b}}-\langle x^{\beta_m},x^{\beta_m} \rangle_{\tilde{b}}=
\langle x^{n},x^{n} \rangle_{\tilde{b}}-\langle x^{m},x^{m} \rangle_{\tilde{b}}.$$
It follows, by (\ref{velika}),
\begin{equation}\label{xbetan}
\|(\mathcal{L}^{x^{\beta_n},x^{\beta_n}}-\mathcal{L}^{x^{\beta_m},x^{\beta_m}})(b)\|< \varepsilon\
\mathrm{const} \|b\|,
\end{equation}
for any $b\geq0$. Since every element of $\mathcal B$ is a linear combination of at most four positive
elements, we conclude that $\L^{x^{\beta_n},x^{\beta_n}}$ is a Cauchy sequence in $L(\mathcal B)$,
multiplying the constant in (\ref{xbetan}) by $4$, if necessary. Hence it converges.

For every $y \in \mathcal{U}$, we have
$$\langle y,x^n-x^m \rangle_{\tilde{b}} \langle x^n-x^m,y \rangle_{\tilde{b}} \leq
\| \langle x^n-x^m, x^n-x^m \rangle_{\tilde{b}} \| \langle y,y \rangle_{\tilde{b}}.$$
By (\ref{b je manje}) and (\ref{Cauchy1}), $\| \langle x^n-x^m,y \rangle_{\tilde{b}} \| < \varepsilon
\sqrt{\| \langle y,y \rangle_1 \|}$, implying
\begin{equation}\label{xbetany}
\|(\mathcal{L}^{x^{\beta_n},y}- \mathcal{L}^{x^{\beta_m},y})(b) \| < \varepsilon \sqrt{\| \langle y,y
\rangle_1 \|} \|b\|,
\end{equation}
for all $\mathcal B\ni b\geq0$. As above, we conclude that $\L^{x^{\beta_n},y}$ is a Cauchy sequence in
$L(\mathcal B)$, and hence convergent. Moreover, it satisfies the Cauchy condition uniformly with respect to
$y$, $\|\skp yy_1\|\leq1$.

Therefore, we proved that there are $K$, $K_y\in L(\mathcal B)$ such that
\begin{equation}\label{jezgro K}
\lim\limits_{n \rightarrow {+\infty}} \| \mathcal{L}^{x^{\beta_n},x^{\beta_n}}-K \|=0,
\end{equation}
\begin{equation}\label{jezgro Ky}
\lim\limits_{n \rightarrow {+\infty}} \| \mathcal{L}^{x^{\beta_n},y}-K_y \|=0.
\end{equation}

Since $\|\mathcal{L}^{x^{\beta_n},x^{\beta_n}}\|,\ \|\mathcal{L}^{x^{\beta_n},y}\|\leq\ \mathrm{const,\ for\
}n \in \mathbb{N}$ and $||\skp yy_1||\le1$, the series
$$\sum\limits_{m=0}^{+\infty} \frac{t^m (\mathcal{L}^{x^{\beta_n},x^{\beta_n}})^m}{m!}\quad\mbox{and}\quad
\sum\limits_{m=0}^{+\infty} \frac{t^m (\mathcal{L}^{x^{\beta_n},y})^m}{m!}$$
uniformly converge with respect to $n\in\mathbb{N}$, which by Lebesgue dominant convergence theorem implies
$$\lim\limits_{n \rightarrow {+\infty}} \langle x^{\beta_n}_t,\bullet x^{\beta_n}_t \rangle=
\lim\limits_{n \rightarrow {+\infty}}e^{t \mathcal{L}^{x^{\beta_n},x^{\beta_n}}}=
\lim\limits_{n \rightarrow {+\infty}} \sum\limits_{m=0}^{+\infty} \frac{t^m (\mathcal{L}^{x^{\beta_n},x^{\beta_n}})^m}{m!}=e^{t K},$$
$$\lim\limits_{n \rightarrow {+\infty}} \langle x^{\beta_n}_t,\bullet y_t \rangle=
\lim\limits_{n \rightarrow {+\infty}}e^{t \mathcal{L}^{x^{\beta_n},y}}=
\lim\limits_{n \rightarrow {+\infty}} \sum\limits_{m=0}^{+\infty} \frac{t^m (\mathcal{L}^{x^{\beta_n},y})^m}{m!}=e^{t K_y}.$$

So,
\begin{equation}\label{velikoOprva}
\lim\limits_{n \rightarrow {+\infty}} \langle x^{\beta_n}_t,\bullet x^{\beta_n}_t \rangle=\mathrm{id}_{\mathcal{B}}+tK+O(t^2),
\end{equation}
\begin{equation}\label{velikoOdruga}
\lim\limits_{n \rightarrow {+\infty}} \langle x^{\beta_n}_t,\bullet y_t \rangle=\mathrm{id}_{\mathcal{B}}+tK_y+O(t^2).
\end{equation}

Thus, we found the kernels of the desired limit of our Cauchy sequence. We can immediately apply Proposition
\ref{LS} to bring up the unit $u_t$ with kernels $K$, $K_y$. However, it is disputable whether or not, this
unit satisfies one of conditions of Proposition \ref{LS}, and therefore, whether or not it belongs to
$\mathcal U$. So, we need to find another way to obtain $u_t$.

Let $\varepsilon>0$. Since limits $\lim\limits_{n \rightarrow {+\infty}} \langle x^{\beta_n}_t,x^{\beta_n}_t
\rangle$ and $\lim\limits_{n \rightarrow {+\infty}} \langle x^{\beta_n}_t,y_t \rangle$ exist in
$\mathcal{B}$, uniformly in $y,\ \| \langle y,y \rangle_1 \| \leq1$, there are $n_1, n_2\in \mathbb{N}$ such
that
\begin{equation}\label{epolaprva}
\|\langle x^{\beta_n}_t, x^{\beta_n}_t \rangle - \langle x^{\beta_{n_1}}_t, x^{\beta_{n_1}}_t \rangle \|<\frac{\varepsilon}{2} \mathrm{\ for\ } n \geq n_1,
\end{equation}
\begin{equation}\label{epoladruga}
\|\langle x^{\beta_n}_t, y_t \rangle - \langle x^{\beta_{n_2}}_t, y_t \rangle \|<\frac{\varepsilon}{2} \mathrm{\ for\ } n \geq n_2.
\end{equation}
Let $n_0=\mathrm{max} \{n_1,n_2 \}$ and $m,n \geq n_0$. We have, by (\ref{epolaprva}) and (\ref{epoladruga})
$$\begin{aligned}\| x^{\beta_n}_t-&x^{\beta_m}_t \|_{E_t}^2=\| \langle x^{\beta_n}_t-x^{\beta_m}_t,x^{\beta_n}_t-x^{\beta_m}_t \rangle
                            \|_{\mathcal{B}}=\\
=&\| \langle x^{\beta_n}_t,x^{\beta_n}_t \rangle - \langle x^{\beta_m}_t,x^{\beta_m}_t \rangle +\\
        &\hskip1cm +\langle x^{\beta_m}_t, x^{\beta_m}_t - x^{\beta_n}_t \rangle +  \langle  x^{\beta_m}_t - x^{\beta_n}_t, x^{\beta_m}_t
            \rangle\| \leq\\
        &\leq \|\langle x^{\beta_n}_t, x^{\beta_n}_t \rangle - \langle x^{\beta_{n_0}}_t, x^{\beta_{n_0}}_t \rangle \|
            + \|\langle x^{\beta_m}_t, x^{\beta_m}_t \rangle - \langle x^{\beta_{n_0}}_t, x^{\beta_{n_0}}_t \rangle \|+\\
        &\hskip1cm+2\|\langle x^{\beta_m}_t, x^{\beta_m}_t \rangle - \langle x^{\beta_{n_0}}_t, x^{\beta_{n_0}}_t \rangle \|
            + 2\|\langle x^{\beta_{n_0}}_t, x^{\beta_{n_0}}_t \rangle - \langle x^{\beta_{n_0}}_t, x^{\beta_n}_t
            \rangle\|+\\
        &\hskip1cm+ 2\|\langle x^{\beta_m}_t, x^{\beta_n}_t \rangle - \langle x^{\beta_{n_0}}_t, x^{\beta_n}_t \rangle\|<8\varepsilon,
\end{aligned}$$
It follows that $(x^{\beta_n}_t)$ is convergent in Hilbert $\mathcal{B}-\mathcal{B}$ module $E_t$. Its limit
we denote by
\begin{equation}\label{ostaje}
\lim\limits_{n \rightarrow {+\infty}} x^{\beta_n}_t=u_t \in E_t.
\end{equation}
By (\ref{velikoOprva}) and (\ref{velikoOdruga}), we get
$$\langle u_t,\bullet u_t \rangle=\mathrm{id}_{\mathcal{B}}+tK+O(t^2),$$
$$\langle u_t,\bullet y_t \rangle=\mathrm{id}_{\mathcal{B}}+tK_y+O(t^2).$$

To conclude that $u_t$ is unit, we only need to apply limit (as $n\to\infty$) to relation
$$x^{\beta_n}_t\otimes x^{\beta_n}_s=x^{\beta_n}_{s+t}.$$
From (\ref{jezgro K}) and (\ref{jezgro Ky}) we find that
$$
\lim\limits_{n \rightarrow {+\infty}} \| \langle x^{\beta_n}-u, x^{\beta_n}-u \rangle_1 \|=0,
$$
i.e.
$$\lim\limits_{n \rightarrow {+\infty}} [x^n]=[u],$$
in $\mathcal{U}/_{\sim}$.
Therefore $\mathcal{U}/_{\sim}$ is a Hilbert $\mathcal{B}$-module.

If $E$ is only subspatial, we can embed it, into a spatial system $\hat E$ with central unit $\omega$. We can
apply the previous case. The only question is whether the limit unit belongs to $E\le \hat E$. However it
immediately follows from (\ref{ostaje}).
\end{proof}

\begin{prp}\label{cepanje niza}
If $E$ is a subspatial product system, then $\mathcal U$ is (algebraically) isomorphic to
$\ind(E)\oplus\mathcal B$ as right $\mathcal B$ module. If $E$ is, in addition, spatial, then $\mathcal U$ is
isomorphic to $\ind(E)\oplus\mathcal B$ as left-right $\mathcal B-\mathcal B$ module.
\end{prp}

\begin{proof}
We can assume that $\omega$ is unital, since the index does not depend on $\omega$.

By the previous Theorem and Proposition, we have the short exact sequence of Hilbert modules.
$$0\rightarrow\mathcal B\overset i\hookrightarrow\mathcal U\overset\pi\rightarrow\ind(E)\rightarrow0$$
where $i(\beta)=\omega^\beta$, and $\pi$ is the canonical projections. We shall show that this sequence
splits, constructing the homomorphism $j:\mathcal U\to\mathcal B$ by
$$j(x)=\L^{x,\omega}(1)^*=\L^{\omega,x}(1).$$
This mapping satisfies
$$j(x+y)=(\L^{x,\omega}(1)+\L^{y,\omega}(1)-\L^{\omega,\omega}(1))^*=j(x)+j(y),$$
$$j(x\cdot a)=\left(a^*\L^{x,\omega}(1)+(1-a^*)\L^{\omega,\omega}(1)\right)^*=j(x)a,$$
since $\omega$ is unital, implying $\L^{\omega,\omega}(1)=0$.

If, in addition, $\omega$ is central, then
$$j(a\cdot x)=\left(\L^{x,\omega}(a^*)+\L^{\omega,\omega}(1-a^*)\right)^*=(\L^{x,\omega}(1)a^*)^*=aj(x),$$
since for central unital unit $\omega$ there holds $\L^{x,\omega}(a)=\L^{x,\omega}(1)a$ (Lemma \ref{svojstva
centralne}).

To finish the proof, note that $j\circ i(\beta)=j(\omega^\beta)=\beta$.
\end{proof}

The following Corollary was proved in \cite[Theorem 3.5.2]{JFA04} under additional assumption that $\mathcal
B$ is a von Neumann algebra and in full generality in \cite[Theorem 5.2]{SK03}. Here, we give an easy proof
that does not use Kolmogorov decomposition of completely positive definite kernels.

\begin{cor} Let $x$ be a continuous unit on some product system $E$ over $\mathcal B$. If $E$ can be embedded
in some spatial product system, then the generator of CPD semigroup $\skp{x_t}{bx_t}$ has the
Christensen-Evans form, that is
$$\L^{x,x}(b)=\skp{\zeta_x}{b\zeta_x}+\beta_x^*b+b\beta_x,$$
where $\zeta_x$ is element of some Hilbert left-right $\mathcal B-\mathcal B$ module, and $\beta_x\in\mathcal
B$.

Moreover, the generator of $\skp{x_t}{by_t}$ has the form
$$\L^{x,y}(b)=\skp{\zeta_x}{b\zeta_y}+\beta_x^*b+b\beta_y.$$
\end{cor}

\begin{proof} Let $E\le\hat E$ and let $\omega$ be a central unital unit in $\hat E$. A straightforward
calculation gives
$$\skp x{b\cdot y}=\L^{x,y}(b)-\L^{x,\omega}(b)-\L^{\omega,y}(b)+\L^{\omega,\omega}(b).$$
Using the fact that $\omega$ is central and unital, we get $\L^{x,\omega}(b)=\L^{x,\omega}(1)b=j(x)^*b$,
$\L^{\omega,y}(b)=b\L^{\omega,y}(1)=bj(y)$ and $\L^{\omega,\omega}(b)=0$, where $j$ is the mapping from
Proposition \ref{cepanje niza}. Thus we obtain
$$\L^{x,y}(b)=\skp x{b\cdot y}+j(x)^*b+bj(y)=\skp{[x]}{b[y]}+j(x)^*b+bj(y),$$
which finishes the proof.\end{proof}

\section{Examples, Remarks}\label{secremarks}

Following two Examples demonstrate that $\ind(E)$ defined in this note is a generalization of the notion of
index defined by Arveson in the case $\mathcal B=\mathrm C$ in \cite{Arv89}, and by Skeide in the case when
$E$ is a spatial product system \cite{KPR06}.

\begin{exm}Let $E$ be an Arveson product system, i.e.~product system over $\mathbb C$. Then $\mathcal U/_{\sim}$
is isomorphic to a vector space of dimension $\mathrm{ind}(E)$. Indeed, as any Arveson product system
contains a unique maximal type I subsystem of the same index (namely, the system generated by its units), we
may assume that $E$ is generated by some continuous set of units. By \cite[Theorem 6.7.1]{Arv03} and
\cite[Proposition 3.1.5]{Arv03}, $E$ is isomorphic to the concrete product system of the CCR flow of rank
$n=\mathrm{ind}(E)$, acting on $\mathcal{B} ( e^{L^2((0,\infty),K) })$ where $K$ is a Hilbert space of
dimension $n$. By \cite[Theorem 2.6.4]{Arv03}, $\mathcal U/_{\sim}=\{[U^{\zeta}],\ \zeta \in K \}$, where
$U^{\zeta}$ are units defined by $U^{\zeta}_t(\mathrm{exp}(f))=\mathrm{exp}(\chi_{(0,t)}\otimes \zeta +
S_tf), \ t\geq 0,\ f \in L^2((0,\infty);K)$, and $(S_t)$ is the shift semigroup of index $n$ that acts on
$L^2((0,\infty);K)$ by way of
$$S_tf(x)=\begin{cases}f(x-t),&\ x>t\\0,&\ 0<x\leq t.\end{cases}$$
Taking the unit $U^0$ for $\omega$ we see that
$\skp{[U^{\zeta}]}{[U^{\eta}]}=\skp{U^{\zeta}}{U^{\eta}}=\skp{\zeta}{\eta}$. Noting that $\skp{
[U^{\zeta+\eta}]-[U^{\zeta}+U^{\eta}]}{[U^{\zeta+\eta}]-[U^{\zeta}+U^{\eta}]}=0$ and
$\skp{[U^{a\zeta}]-[aU^{\zeta}]}{[U^{a\zeta}]-[aU^{\zeta}]}=0$, for $a\in \mathbb{C}$, we obtain that
$[U^{\zeta+\eta}]=[U^{\zeta}+U^{\eta}]$ and $[U^{a\zeta}]=[a U^{\zeta}]$. Hence, $K\ni\zeta
\mapsto[U^{\zeta}]\in \mathcal U/_{\sim}$ is an isomorphism.

Since the dimension of vector space completely determines it (up to isomorphism), it allows us to consider
the $\mathcal B$-module $\ind(E)$ as a suitable generalization of index.
\end{exm}

\begin{exm}Let $\G(F)$ be the time ordered Fock module where $F$ is a two sided Hilbert module
over $\mathcal B$. In \cite[Theorems 3 and 6]{LS01} it was proved that all continuous units in $\G(F)$ can be
parameterized by the set $F\times\mathcal B$. The unit that corresponds to pair $(\zeta,\beta)$ denote by
$u(\zeta,\beta)$. The corresponding kernels are given by (see \cite[formula (3.5.2)]{JFA04})
\begin{equation}\label{Fock kernel}
\L^{u(\zeta,\beta),u(\zeta',\beta')}(b)=\skp\zeta{b\zeta'}+\beta^*b+b\beta'
\end{equation}
Comparing the kernels, we can conclude that the mapping $F\times\mathcal B\ni(\zeta,\beta)\to
u(\zeta,\beta)\in\mathcal U_{\G(F)}$ is an (algebraic) isomorphism of modules, if we choose $\omega=u(0,0)$.
Also, it is easy to see that $u(\zeta,\beta)^\gamma=u(\zeta,\beta+\gamma)$, so that $\mathcal
U_{\G(F)}/\sim=\{[u(\zeta,0)]\;|\;\zeta\in F\}$ and therefore $\ind(\G(F))\cong F$, in algebraic sense.

Further, from (\ref{Fock kernel}) we easily get
$$\skp{u(\zeta,\beta)}{u(\zeta',\beta')}=\skp\zeta{\zeta'}.$$

Thus $\ind(\G(F))$ is isomorphic to $F$ as Hilbert left-right module. Therefore, our definition of index
generalizes that of Skeide \cite{KPR06}.
\end{exm}

\begin{rmr} It would be interesting to compute the index of subspatial system exhibited in \cite[Section
3]{PAMS10}, that is not spatial.
\end{rmr}

Next Example is the Example of a product system without any central unit.

\begin{exm}In \cite[Example 4.2.4]{JFA04}, there is an example of product system that does not contain any
central unit. In more details, let $\mathcal{B}=K(G)+\mathbb{C} {\id}_G$ be the unitization of compact
operators on some infinite-dimensional Hilbert space $G$ and let $h\in B(G)$ be a self-adjoint operator. The
Hilbert $\mathcal{B}-\mathcal{B}$ modules $\mathcal{B}_t$ defined to coincide with $\mathcal{B}$ as right
Hilbert modules and with left multiplication $b\cdot x_t=e^{ith}be^{-ith}x_t$ form a product system
$(\mathcal{B}_t)_{t\ge0}$ with identification $x_s\otimes y_t=e^{ith}x_s e^{-ith}y_t$. Such product system
does not admit a central unit if and only if $h\notin\mathcal B$, and it is generated by the single unit
$1_t\equiv 1$ and hence it is of type $I$, as it was shown in \cite[Example 4.2.4]{JFA04}.

Let $\xi_t$ be an arbitrary continuous unit and let $\xi'_t=e^{-ith}\xi_t$. Obviously, $\xi'_t$ is uniformly
continuous family. We have
$$\xi'_{s+t}=e^{-i(s+t)h}\xi_s\otimes\xi_t=e^{-i(s+t)h}e^{ith}\xi_se^{-ith}\xi_t=\xi'_s\xi'_t.$$
It follows that $\xi'_t=e^{tB_\xi}$ for some bounded operator $B_\xi$ on $G$. Set $A_\xi=B_\xi+ih$ and we
obtain that any continuous unit on $(\mathcal B_t)_{t\ge0}$ has a form
$$\xi_t=e^{ith}e^{t(A_{\xi}-ih)},$$
for some $A_{\xi}$. Moreover, we find that
$$A_\xi=\lim_{t\to0+}\left[\frac{e^{ith}-1}te^{t(A_\xi-ih)}+\frac{e^{t(A_\xi-ih)}-1}t\right]=
\lim_{t\to0+}\frac{e^{ith}e^{t(A_\xi-ih)}-1}t\in\mathcal B,$$
because the last fraction belongs to $\mathcal B$, and the limit converges uniformly, since exponentials are
analytic functions.

Pick the unit $\omega$ choosing $A_{\omega}=0$. Then $\omega_t=1$. Noting that
$\L^{\xi,\eta}={\id}_{\mathcal{B}}(A_{\eta}-ih)+(A_{\xi}^*+ih){\id}_{\mathcal{B}}$, we see that
$\skp{[\xi]}{[\eta]}=\skp{\xi}{\eta}=0$ for every $[\xi],[\eta]\in\mathcal U/_{\sim}$. Hence, $\mathcal
U/_{\sim}=\{0\}$.
\end{exm}

\begin{rmr}
This example shows that product systems, even of type $I$ can not be classified by its index. Namely, for
$h\in\mathcal B$ it has a central unit, and for $h\notin\mathcal B$ it has not. Therefore such product
systems are not isomorphic, in spite of the fact that they have the same index.

In \cite[Theorem 4.8]{PAMS10} it was shown that the product system from previous Example is not subspatial,
by proving that the kernel of the unit $1$ has no Christensen-Evans form. This is, up to our knowledge, the
only example of nonsubspatial product system. However for this system, Propositions \ref{sim-approx},
\ref{cepanje niza} and Theorem \ref{kompletnost} remains valid. This example is trivial, (despite of twisted
left action of $\mathcal B$) in the sense that all $E_t$ are isomorphic (as right modules) to the algebra
$\mathcal B$ itself. On the other hand, index is constructed to "measure" how many dimensions the continuous
units can generate (in a certain sense) after taking quotient by $\mathcal B$.

It is, therefore, natural to ask are there product systems (of course that are not subspatial) for which
Propositions \ref{sim-approx}, \ref{cepanje niza} and Theorem \ref{kompletnost} fails.
\end{rmr}

\begin{rmr}The other problem that arises from this example is what is actually, the trivial product system.
Following Skeide, it must be this example with $h=0$ (as well as any example where $E_t\cong\mathcal B$ - the
algebra itself, and left and right multiplication are canonical). However, then we have problem how to define
short exact sequences in the category of product systems. Namely, injective morphisms have trivial kernel,
i.e.\ isomorphic to $\{0\}$ at each fiber. Such product system has no continuous units, since it must be
equal to $1$ at time $t=0$. Thus, we are forced to consider injective morphisms modulo trivial systems. In
this case, the previous Example can not be seen from short exact sequences. Therefore, in the absence of
suitable definition we can not speak about exact functoriality of index.
\end{rmr}

\end{document}